\documentclass{amsart}

\usepackage{graphicx} 
\usepackage{amsmath,amssymb,amsthm,tikz}
\usepackage{xcolor}
\usepackage{esint}
\usepackage{pgfplots}
\usepackage{scalerel}
\pgfplotsset{compat=1.18}
\usepackage{centernot}

\usepackage{hyperref}
\hypersetup{
colorlinks,
urlcolor=azul,
linkcolor=azul,
citecolor=azul,
} 

\definecolor{azul}{RGB}{33,64,154}

\def \R {\mathbb{R}}

\def \M {\text{Sym}}
\def \I {\mathcal{I}}
\def \J {\mathcal{J}}
\def \MM {\mathcal{M}}

\newtheorem{theorem}{Theorem} 
\newtheorem{lemma}{Lemma}
\newtheorem{proposition}{Proposition}  

\newtheorem{definition}{Definition}
\newtheorem{remark}{Remark}

\numberwithin{equation}{section}

\title[On fractional quasilinear equations with elliptic degeneracy]{On fractional quasilinear equations with \\ elliptic degeneracy}

\author[D.J. Ara\'ujo]{Dami\~ao J. Ara\'ujo}
\address{Department of Mathematics, Universidade Federal da Para\'iba, 58059-900, Jo\~ao Pessoa-PB, Brazil}{}
\email{araujo@mat.ufpb.br}

\author[D. dos Prazeres]{Disson dos Prazeres}
\address{
Departamento de Matem\'atica , Universidade Federal de Sergipe,
Sala 19, São Cristóvão-SE, Brazil}{}
\email{disson@mat.ufs.br}

\author[E. Topp]{Erwin Topp}
\address{
Instituto de Matemáticas, Universidade Federal do Rio de Janeiro, Rio de Janeiro - RJ, 21941-909, Brazil, and Departamento de Matem\'atica y C.C., Universidad de Santiago de Chile, Casilla 307, Santiago, Chile}{}
\email{etopp@im.ufrj.br; erwin.topp@usach.cl}

\begin{document}



\begin{abstract}
In this work, we present a systematic approach to investigate the existence, multiplicity, and local gradient regularity of solutions for nonlocal quasilinear equations with local gradient degeneracy. Our method involves an interactive geometric argument that interplays with uniqueness property for the corresponding homogeneous problem, leading with gradient H\"older regularity estimates. This approach is intrinsically developed for nonlocal scenarios, where uniqueness holds for the local homogeneous problem. We illustrate our results by showing classes of exterior data that exhibit multiple solutions, while also highlighting relevant cases where uniqueness is confirmed.
\end{abstract}  

\date{\today}

\maketitle

\tableofcontents

\section{Introduction}\label{secIntroduction}

The mathematical analysis of nonlocal equations constitutes a relevant chapter in the modern theory of partial differential equations, being, in the last two decades, the subject of an exhaustive study due to its presence in an ample class of mathematical frameworks, and also due for its applications in different models in natural sciences and engineering. Since the paper of Caffarelli and Silvestre~\cite{CS}, we have learned that this is the paradigm of what is nowadays understood as an \textsl{uniformly elliptic nonlocal operator}. 

Despite the importance of this class of operators, many mathematical models involve prototypes whose ellipticity degenerates along an \textit{a priori} unknown region, that might depend on the solution itself.  The corresponding scenario for the non-divergence form of second-order elliptic equations with gradient degeneracy deals with equations of the type 
$$
\mbox{trace}(A(x,Du)D^2u)=f \quad \mbox{in } \; \Omega \subset \mathbb{R}^N,
$$
for diffusion matrices $A:\Omega \times \mathbb{R}^N \to \M(N)$ satisfying 
$$
\lambda |\xi|^\gamma  I_{N} \leq A(x,\xi) \leq \Lambda |\xi|^\gamma I_{N},
$$
with $\gamma > 0$ and \textit{ellipticity} constants $0<\lambda\leq \Lambda$. Here, we denote $\M(N)$ the set of symmetric $N \times N$ matrices, and $I_{N}$ the identity matrix. The equation above degenerates along the set of critical points $\mathcal{S}(u):=\{x \in \Omega \,: \, |Du(x)|=0\}$ of a given solution $u$. Compared with the uniformly elliptic case, this fact imposes less efficient diffusion features near such a region, and therefore the study of analytical and geometric properties for corresponding equations becomes more delicate.

The associated non-local scenario corresponds to the integro-differential operator
$$
Lu(x)=\frac{1}{2}\int_{\mathbb{R}^N}\frac{u(x+a(Du(x)) \cdot z)+u(x-a(Du(x))\cdot z)-2u(x)}{|z|^{N+2s}}dz, 
$$
where $A(x,\xi)=a(\xi)I_{N\times N}$, see \cite{ChJ}. Particularly, 
\begin{equation}\nonumber
Lu=|Du|^\gamma \Delta^s u,
\end{equation}
arises as the main prototype for the former class, where, for parameters $s \in (0,1)$, $\Delta^s$ denotes the fractional Laplacian of order $2s$:
$$
\Delta^s u(x) := -(-\Delta)^s u(x) = C_{N,s} \mathrm{P.V.} \int_{\R^N} \frac{u(y) - u(x)}{|x - y|^{N + 2s}}dy.
$$
$\mathrm{P.V.}$ stands for the Cauchy Principal Value, and $C_{N, s} > 0$ the normalizing constant, see~\cite{Hitch}. In particular, this leads to the stable convergence $\Delta^s \to \Delta$, by letting $s \uparrow 1$. 

\smallskip

In this paper, we are interested in Dirichlet problems with the form
\begin{equation}\label{eq}
\left \{ \begin{array}{rll} |Du|^\gamma \Delta^s u & = f \quad & \mbox{in} \ \Omega \\
u & = g \quad & \mbox{in} \ \mathbb{R}^N\setminus \Omega. \end{array} \right .
\end{equation}
Here, $\Omega \subset \R^N$ is a bounded domain with $C^2$ boundary, the source term $f$ is continuous and bounded in $\Omega$, and the exterior data $g$ is continuous in $\mathbb{R}^N\setminus \Omega$, with adequate growth at infinity.

The analogous local equation
\begin{equation}\label{eqloc}
|Du|^\gamma \Delta u  = f
\end{equation}
was extensively studied during the last two decades, see for instance~\cite{BB,BD,IS}. Existence of viscosity solutions for \eqref{eqloc} and more general degenerate elliptic problems can be obtained by stability through the vanishing viscosity method. Uniqueness is more complicated and requires further assumptions on the source term $f$. This is the case when $f$ is constant, see \cite{BB,BD}; $f$ is strictly positive or strictly negative in $\Omega$, see \cite{J,LW}. 

A breakthrough of \cite{IS} establishes local $C^{1,\alpha}$ regularity estimates for solutions to \eqref{eqloc}, for some universal exponent $\alpha \ll 1$. Optimal interior regularity estimates is proved in \cite{ART}, for
\begin{equation}\label{alpha1}
\alpha= \frac{1}{1+\gamma}.
\end{equation}
See also \cite{AS} for improved boundary regularity estimates. For this, the \textit{canceling property}
$$
|Du|^\gamma \Delta u = 0 \quad \Longleftrightarrow \quad  \Delta u = 0,
$$
see \cite[Section 5]{IS}, is crucially used, since the homogeneous case ($f \equiv 0$) enjoys higher regularity. This feature plays a crucial role in obtaining interior gradient oscillation estimates for solutions of~\eqref{eqloc}. 

In contrast with the local case, even fixed exterior data $g$, nonlocal homogeneous equations like~\eqref{eq} may present multiple solutions, and so, solutions may not be $s$-harmonic, see examples in Section \ref{sec-examples} (and also Figure \ref{fig:Kariri} below).

\begin{figure}[h]
\centering
\begin{tikzpicture}[scale=1.0]
\begin{axis}[
%    hide axis,
    color=gray,
    very thin,
    xtick = {-1,1},
    ytick = {0},
%    axis lines = left,
%    xlabel = \(x\),
%    ylabel = {\(f(x)\)},
]
%function1
\addplot [
    domain=-4.0:-1.0, 
    samples=100, 
    color=black!60,
    very thick,
]
{((1-x)^(2/3))/(2^(2/3))};
%function2
\addplot [
    domain=-1.0:1.0, 
    samples=100, 
    color=black!20!white,
    very thick,
    ]
    {1} node[above, midway] {$v_s$};
%function3
\addplot [
    domain=1.0:4.0,  
    samples=100, 
    color=black!60,
    very thick, 
]
{((1+x)^(2/3))/(2^(2/3))} node[right] {$g_s$};
%function4
\addplot [
    domain=-1.0:1.0, 
    samples=100, 
    color=gray!70!white, 
    very thick,
    ]
    {((1-x)^(2/3))/(2^(2/3))+((1+x)^(2/3))/(2^(2/3))} 
    node[above, midway] {$u_s$};
\end{axis}
\end{tikzpicture}
\caption{For the exterior data $g_s(x)=(1+x)_+^s+(1-x)_+^s$ defined in $\mathbb{R}\setminus (-1,1)$, we observe that $u_s(x)=(1+x)_+^s+(1-x)_+^s$ is the $s$-harmonic in $(-1,1)$. However, $v_s(x)=2^s$ also solves \eqref{eq} for $f=0$. In parallel, according to the uniqueness property for the local scenario, we observe that functions $u_s(x)$ and $v_s(x)$ converge to $u_1(x)=2$, as $s\uparrow 1$.}
\label{fig:Kariri}
\end{figure}
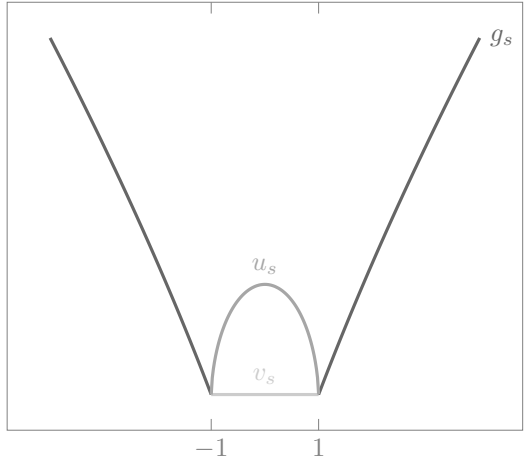

Since the lack of uniqueness for the homogeneous case ($f\equiv 0$) is equivalent to  
$$
|Du|^\gamma \Delta^s u = 0 \quad \centernot\implies \quad  \Delta^s u = 0 ,
$$
we distinguish whenever we mention a $s$-harmonic function -- classically $\Delta^s u = 0$ -- and a \textit{degenerate $s$-harmonic function} -- solutions for $|Du|^\gamma \Delta^s u = 0$ 
. Notice that $s$-harmonic functions are in particular degenerate $s$-harmonic. Also, using the language of viscosity solutions, we observe that, for any $\gamma>0$, the following equivalence holds
$$
|Du|^\gamma \Delta^s u = 0 \; \mbox{ in } \Omega \quad \Longleftrightarrow \quad  \Delta^s u = 0  \; \mbox{ in } \{|Du|\neq 0\} \cap \Omega.
$$
The later relates to non-local elliptic problems with patches of zero gradient, see \cite{CS1,CSS} for local nonlinear scenarios.
In view of this, we denote the nonempty set of degenerate $s$-harmonic functions by $\mathcal{H}(s, g, \Omega)$. 

In this paper, we shall focus on scenarios where the problem
\begin{equation}\label{eqhom}
\left \{ \begin{array}{rll} |Du| \Delta^s u & = 0 \quad & \mbox{in} \ \Omega \\
u & = g \quad & \mbox{in} \ \mathbb{R}^N\setminus \Omega. \end{array} \right .
\end{equation}
has unique solution. Since this feature depends on $g:\mathbb{R}^N \setminus \Omega \to \mathbb{R}$, we define 
$$
\mathcal{P}(\Omega):=\{g: \mathbb{R}^N \setminus \Omega \to \mathbb{R} \; | \; \mbox{problem } \eqref{eqhom} \mbox{ has unique solution} \}.
$$ 

To address these difficulties, we provide an interactive geometric argument that effectively exploit  uniqueness properties of the corresponding homogeneous problem. 
More specifically, we apply the tangential analysis method introduced by L. Caffarelli, in such a way that uniqueness of the approximating problems is preserved step by step in dyadic balls. This approach is properly designed for nonlocal scenarios where, as previously mentioned, the behavior of solutions of homogeneous equations differs from their local counterparts.

\smallskip

Our first main theorem concerns interior optimal regularity estimates. 

%%%%%
\begin{theorem}\label{teoregintro}
Let $s\in (\frac{1}{2},1)$, $\gamma > 0$, $f \in L^\infty(\Omega)$, and assume there  exists $M > 0, \sigma \in (0,2s)$ such that
$
|g(x)| \leq M(1 + |x|)^\sigma 
$
for $x \in \mathbb{R}^N\setminus \Omega$. Assume further that
$g \in \mathcal{P}(\Omega)$. 
Let $u$ be a viscosity solution for \eqref{eq}. Then $u$
is locally $C^{1,\alpha}$ in $\Omega$, for $\alpha=\alpha(\gamma,s)$ given by
\begin{equation}\label{alphas}
\alpha(\gamma,s)=\frac{2s-1}{1+\gamma}.
\end{equation}
\end{theorem}
%%%%%

\medskip

Note that exponent $\alpha$ in \eqref{alpha1}, related to the local scenario, is exactly derived as
$$
\alpha(\gamma,1):=\lim_{s \uparrow 1} \alpha(s,\gamma).
$$
In Section \ref{secreg}, the above theorem is presented in a more general fully nonlinear context, see Theorem~\ref{teoreg}, in which, interior $C^{1, \beta}$ estimates are obtained for the exponent 
$$
\beta = \min \Big{\{} \alpha_K^-,\frac{2s-1}{1+\gamma} \Big{\}},
$$ 
where $C^{1,\alpha_K}$ is regularity for the harmonic function associated with the nonlocal uniformly elliptic operator assumed in \eqref{LATAM}, and $\alpha_K^-$ stands any positive number strictly less the $\alpha_K$. Since $s$-harmonic functions are classical, Theorem~\ref{teoregintro} is obtained, in particular, for $\beta=\alpha(\gamma,s)$ as in \eqref{alphas}. For some $C=C(\gamma,s,N)$, the function
$$
C|x|^{\frac{2s+\gamma}{1+\gamma}},
\quad x \in \mathbb{R}^N
$$
solves \eqref{eq} for $f\equiv 1$, and so, the optimality of exponent $\alpha$ in \eqref{alphas} can be verified. We also observe that the same optimal exponent can be obtained for convex operators for which higher regularity estimates are available, see~\cite{CSAnnals, S}. 

In view of Theorem~\ref{teoregintro}, exterior condition $g$ emerges as a decisive ingredient for detecting uniqueness of degenerate $s$-harmonic functions. Exterior datum $g\equiv constant$ are emblematic, where by the strong maximum principle, we guarantee  that $g\in \mathcal{P}(\Omega)$. Other relevant examples are provided in Section~\ref{sec-examples}. See \cite{OS} for a similar context, in which interior regularity relies on choosing appropriate exterior datum. 

Our second main result establishes qualitative properties for multiplicity of solutions to the homogeneous case. In view of this, we introduce the standard notation: for a given measurable set $A \subset \R^N$ and parameter $\sigma \in (0,2)$, we shall denote the space
\begin{equation*}
L^1_{2s}(A) := \Big{\{} u \in L^1_{loc}(A) : \int_{A} \frac{|u(y)|}{1 + |y|^{N + 2s}}dy < +\infty \Big{\}}.
\end{equation*}

\smallskip

Our second main result is the following:

%%%%%%%%
\begin{theorem}\label{teomaxmin}
Let $s \in (0,1)$ and $\gamma > 0$. Assume $\Omega \subset \R^N$ a bounded domain with $C^2$ boundary, $g \in L^1_{2s}(\mathbb{R}^N\setminus \Omega) \cap C(\mathbb{R}^N\setminus \Omega)$, and $w$ the unique solution of problem \eqref{eqhom}. The following statements hold:

\begin{enumerate}
\item[(i)]  There exist $\bar u, \underline u \in \mathcal{H}(s,g,\Omega)$, such that for each $u \in \mathcal{H}(s,g,\Omega)$, 
$$
\underline u \leq u \leq \bar u.
$$

\item[(ii)] Assume $\bar u$ (resp. $\underline u$) touches $w$ at some point $x \in \Omega$, then 
$$
\bar u = w \; \mbox{(resp. $\underline u = w$)}.
$$

\item[(iii)] Assume $\bar u$ touches $\underline u$ at some point $x \in \Omega$, then 
$$
g \in \mathcal{P}(\Omega).
$$

\item[(iv)] Assume  $|Dw|\neq0$ in $\Omega$, then 
$$
g \in \mathcal{P}(\Omega).
$$

\item[(v)] If $w$ has a strict local maximum (resp. strict local minimum) in $\Omega$. Then, 
$$
\underline u < w \; \mbox{(resp. $\bar u > w$)}.
$$
\end{enumerate}
\end{theorem}
%%%%%

Hereafter in this paper, we call $\underline u, \bar u$ (bounded in $\Omega$) \textsl{the minimal and maximal} functions for $\mathcal{H}(s,g,\Omega)$. The maximal solution is constructed letting $\eta \uparrow 0$, in the sequence of solutions $u^\eta$ for the problem \eqref{eq}, with $f=-\eta$. The Minimal solution is obtained in an analogous way. Item (v) is a purely nonlocal phenomena. As it can be observed in~\cite{DSV}, $s$-harmonic functions may present strict local extrema, in contrast with the second-order case, in which the strong maximum (minimum) principle holds. In relation to (iv), it is easy to see that the converse is not true: for $g\equiv constant$, we have $g \in \mathcal{P}(\Omega)$, but $|Dw|= 0$ in $\mathbb{R}^N$.

The results in Theorem \ref{teomaxmin} shall be presented in a more general fully nonlinear version, see Theorem~\ref{teomaxmingeral}. Furthermore, we believe that these results can also be extended to other fractional quasilinear models that share some structural similarities, such as the nonvariational fractional $p$-Laplacian
$$
\Delta_p^s u(x) = C_{N, s} \int_{\R^N} [u(x + j_p(Du(x))z) - u(x) - \chi_B(z) j_p(Du(x)) \cdot z] \frac{dz}{|z|^{N + 2s}},
$$
where $\chi$ denotes the indicator function and $j_p$ the jump function  
\begin{equation}
j_p(q) =
\left\{
\begin{array}{ccl}
 |q|^{\frac{p-2}{2}}(I_N + r_p \hat q \otimes \hat q) & \mbox{for} & q \neq 0,\\[0.2cm]
 0 & \mbox{for} & q=0,
\end{array}
\right.
\end{equation}
for $p>2$, where $\hat q = q/|q|$, $I_N$ is the identity $N \times N$ matrix, and $r_p > 0$ is a normalizing constant. This operator is of particular interest since $\Delta_p^s u \to \Delta_p u$ as $s \uparrow 1$, where $\Delta_p$ is the second-order $p$-Laplace operator, see~\cite{ChJ}. Nevertheless, we do not know if they can be applied to the variational version of the fractional $p$-Laplacian, which naturally arises in the study of critical points of energy functionals defined on fractional Sobolev spaces $W^{s,p}$. Considering this type of operators, the equivalence among weak (distributional) and viscosity solutions of homogeneous equations is proven in~\cite{BM}, as a nonlocal counterpart of the analogous result for the second-order $p$-Laplace operator provided in~\cite{JLM}.

We also expect that the methods used in the proof of Theorem \ref{teomaxmin} might be applied to singular (rather than degenerate) quasilinear nonlocal operators as the one presented in~\cite{BCF}, associated to fractional ``tug-of-war" games (whose local limit is the normalized $p$-Laplacian). In fact, the existence of the extremal solutions for these operators (in the sense of Theorem~\ref{teomaxmin}) was already mentioned in the survey paper~\cite{C}.

We stress that, for general source term $f$, we do not know if the solution to~\eqref{eq} is unique, even if the associated homogeneous problem has unique solution. If $f$ does not vanish in $\Omega$, problem~\eqref{eq} has a unique solution (for general admissible $g$), see Theorem~\ref{teo1} below. 

We conclude the introduction by pointing out some open questions that emerged during the preparation of this work. The main problem is related to interior $C^{1, \alpha}$ estimates for solutions of~\eqref{eq}, for general exterior data $g$. Even for the homogeneous case, this seems to require new ideas, and this would be of great benefit to deal with the non-homogeneous problem. We also think that a classification result for exterior data $g$ belonging to $\mathcal{P}(\Omega)$ is of great interest. This would led to a better understanding of the problem, as well as identifying new classes of solutions, different to extremal and $s$-harmonic functions.

\smallskip

The paper is organized as follows. In Section \ref{prelim}, we establish the general context that we treat in this paper as definition and basic results, comparison principle, and strong maximum principle. In Section \ref{secinh}, we discuss existence, uniqueness, and Holder regularity for the Dirichlet problem related to equation \eqref{eq}. In Section \ref{homsec}, we prove Theorem \ref{teomaxmin}. In Section \ref{sec-examples}, we explore examples that present uniqueness and non-uniqueness for the homogeneous equation. In the last Section, we prove the Theorem \ref{teoreg}.

%%%%%%%%
\section{Preliminaries }\label{prelim}

Throughout this paper, $\Omega$ shall be a fixed bounded domain in $\mathbb{R}^N$ with $C^2$ boundary, the source term $f$ is assumed continuous and bounded in $\Omega$. For $s \in (0, 1)$, we denote $\mathcal L_0$ the family of symmetric kernels $K: \R^N \setminus \{ 0 \} \to \R_+$ satisfying 
\begin{equation}\label{elliptic}
\lambda |z|^{-(N + 2s)} \leq K(z) \leq \Lambda |z|^{-(N + 2s)}, \quad z \neq 0
\end{equation}
for certain constants $0 < \lambda \leq \Lambda < + \infty$. Given $K \in \mathcal L_0$, $u: \R^N \to \R$ measurable and $x \in \R^N$, we consider linear operators $L_K$ with the form
\begin{align*}
L_K u(x) = & \mathrm{P.V.} \int_{\R^N} [u(y) - u(x)] K (x - y) dy \\
= & \mathrm{P.V.} \int_{\R^N} [u(x + z) - u(x)] K (z) dz,
\end{align*}
which is well defined for functions $u \in C^{1,1}(B_\delta(x)) \cap L^1_{2s}(\R^N)$ for some $\delta > 0$. For this, we shall assume the two-parametric family of kernels $\{ K_{ij} \}_{i \in \mathcal I, j \in \mathcal J} \subset \mathcal L_0$ as the class of nonlocal operators 
\begin{equation}\label{operator}
I u(x) = \inf_{i \in \I} \sup_{j \in \J} L_{ij} u(x),
\end{equation}
where we have adopted the notation $L_{ij} = L_{K_{ij}}$ for simplicity. In addition, the Pucci extremal operators associated to the family are given by
$$
\mathcal M^+ u(x) = \sup_{K \in \mathcal L_0} L_{ij} u(x); \quad \mathcal M^- u(x) = \inf_{K \in \mathcal L_0} L_{ij} u(x).
$$

We are interested in Dirichlet problems of the form
\begin{equation}\label{eq1}
\left \{ \begin{array}{rll} |Du|^\gamma I u & = f \quad & \mbox{in} \ \Omega, \\[0.15cm] u & = g \quad & \mbox{in} \ \mathbb{R}^N\setminus \Omega. \end{array} \right .
\end{equation}

Next, according to~\cite{BI, CS}, we define the notion of viscosity solutions for equations as in \eqref{eq1}. Concerning Dirichlet problems involving a prescribed exterior condition in $\mathbb{R}^N\setminus \Omega$, we refer to~\cite{BChI1} for a notion of solution. For this, we denote $\mathrm{USC}$ (resp. $ \mathrm{LSC}$) for upper semicontinuous functions (resp. lower semicontinuous function). 

\begin{definition}
A function  $u \in \mathrm{USC}(\Omega) \cap L^1_{2s}(\R^N)$ (resp. $u \in \mathrm{LSC}(\Omega) \cap L^1_{2s}(\R^N)$) is a viscosity subsolution (resp. supersolution) to
\begin{equation}\label{tamimi}
|Du|^\gamma Iu = f
\end{equation}
at $x_0 \in \Omega$ if for each $\delta > 0$ and each smooth, bounded function $\varphi$ such that $u - \varphi$ attains a maximum (resp. minimum) in $B_\delta(x_0)$ at $x_0$, then
\begin{equation}\label{eqdef}
|D\varphi(x_0)|^\gamma \inf_{i \in \I} \sup_{j \in \J}(L_{ij}[B_\delta] \varphi(x_0) + L_{ij}[B_\delta^c] u(x_0)) \geq \ (\mbox{resp.} \ \leq) \ f(x_0),
\end{equation}
where, for measurable $E \subseteq \R^N, x \in \R^N$, $\phi: \R^N \to \R$ measurable, and $K \in \mathcal L_0$ we have adopted the notation 
\begin{align*}
L_{K}[E] \phi(x) & = \mathrm{P.V.} \int_{E} [\phi(x + z) - \varphi(x)] K(z)dz,
\end{align*}
whenever the integral makes sense.

A given function $u \in C(\Omega) \cap L_{2s}^1(\R^N)$ is a viscosity solution to~\eqref{eqdef} if $u$ is a viscosity sub and supersolution simultaneously.
\end{definition}

For the reader's convenience, we state the ultimate local regularity result available for viscosity solutions of \eqref{eq1}, in which one shall be useful for the purposes of Section \ref{secreg}.

\begin{theorem}\label{Local Holder estimates}[Local H\"older estimates \cite{DPT}]\label{PTreg} Assume $f \in C(\Omega) \cap L^\infty(\Omega)$ and  $g \in C(\mathbb{R}^N\setminus \Omega) \cap L_{2s}^1(\mathbb{R}^N\setminus \Omega)$. Then, viscosity solutions $u \in C(\R^N)$ for~\eqref{eq1} are locally $C^{\mu}(\Omega)$ for universal $\mu \in (0,1]$, such that: $\mu=2s$, in case $s < 1/2$; $\mu$ is any positive number strictly less than $1$, for $s=1/2$; $\mu=1$, in case $s > 1/2$.
\end{theorem}

\section{The inhomogeneous problem}

\subsection{Comparison Principle}

We now deliver a proof for comparison principle for the Dirichlet problem \eqref{eq1} under the assumption that $f$ is strictly away from zero. We carry out the details for the reader’s convenience.

\begin{proposition}\label{comparison}
Assume $\inf_{\Omega} | f | > 0$. Let $u,v$ be respectively viscosity subsolution and supersolution to the equation~\eqref{tamimi} such that $u \leq v$ in $\mathbb{R}^N\setminus \Omega$. Then, $u \leq v$ in $\Omega$.
\end{proposition} 

\begin{proof}
With no loss of generality, we assume $\inf_{\Omega}  f  =: c_f > 0$. Let us suppose, for the purpose of contradiction, that 
$$
M := \sup_{\Omega}\{ u - v \} > 0.
$$
By the upper semicontinuity of $u-v$, we have the existence of $x_0 \in \Omega$ such that $M = (u - v)(x_0)$. For each $\mu >1$ sufficiently close to 1, we have $M_\mu := \sup_{\Omega}\{ \mu u - v \} > 0$, and the supremum is attained at some point to $x_\mu \in \Omega$. We have that any of these points attaining $M_\mu$ is uniformly away from the boundary. In fact, if $x_\mu \to \partial \Omega$ as $\mu \to 1$, since $M_\mu \to M$, we would have $M \leq 0$, which is a contradiction. We denote by $\rho \in (0,1)$ this distance, which can be chosen independent of $\mu$ when this parameter is close to $1$.

Notice that the function $x \mapsto \mu u(x)$ satisfies
$$
|Du|^\gamma I u \geq \mu^{\gamma + 1} f \quad \mbox{in} \ \Omega,
$$
in the viscosity sense. 

For $\epsilon \in (0,1)$, we double variables and consider the function
\begin{equation}\label{doubling}
\Phi(x,y) := \tilde u(x) - v(y) - \frac{1}{2\epsilon^{2}}|x - y|^2, \quad x, y \in \Omega.
\end{equation}
for $\tilde u = \mu u$.
Function $\Phi$ is upper semicontinuous in $\bar \Omega \times \bar \Omega$ and attains its maximum at some point $(\bar x, \bar y)$. Using the inequality $\Phi(\bar x, \bar y) \geq \Phi(x_\mu, x_\mu)$, and by standard arguments, we see that
$$
\epsilon^{-2} |\bar x - \bar y|^2 \to 0 \quad \mbox{and so,} \quad  \bar x, \bar y \to \tilde x_\mu, \quad \mbox{as} \ \epsilon \to 0,
$$
where $\tilde x_\mu \in \Omega$ is such that $(\tilde u - v)(\tilde x_\mu) = M_\mu$.

Then, denoting $\phi(x,y) = \frac{1}{2\epsilon^{2}}|x - y|^2$ and $\bar p = 2(\bar x - \bar y)/\epsilon^2$, we use the viscosity inequalities for $\mu u$ and $v$, for each $\delta > 0$ and $a \in (0,1)$, there exist $i \in \mathcal I, j \in \mathcal J$ such that
\begin{equation}\label{testing}
\begin{array}{rcl}
|\bar p|^\gamma \,(\,L_{ij}[B_\delta] \phi(\cdot, \bar y) (\bar x) + L_{ij}[B_\delta^c] \tilde u(\bar x)) & \geq & \mu^{1 + \gamma} f(\bar x) - a/2, \\[0.2cm] 
|\bar p|^\gamma (-L_{ij}[B_\delta] \phi(\bar x, \cdot) (\bar y) + L_{ij}[B_\delta^c] v(\bar y))  & \leq & f(\bar y) + a/2.
\end{array}
\end{equation}

By the assumption on $f$ and since $\mu$ is away zero, we can choose $a$ small enough in order $\mu^{1 + \gamma} f(\bar x) - a/2 > 0$ for all $\epsilon$. This implies that $\bar p \neq 0$ for all $\epsilon$. We also have that 
$$
|L_{ij}[B_\delta] \phi(\cdot, \bar y) (\bar x)|, |L_{ij}[B_\delta] \phi(\bar x, \cdot) (\bar y)| \leq C \epsilon^{-2} \delta^{2 - 2s},
$$
where $C > 0$ depends only on the ellipticity constants, $N$ and $s$.

Subtracting the viscosity inequalities and using the continuity of $f$, we get that
\begin{equation}\nonumber
\begin{array}{c}
|\bar p|^\gamma \Big{(} C\epsilon^{-2} \delta^{2 - 2s} + L_{ij}[B_\delta^c] \tilde u(\bar x) - L_{ij}[B_\delta^c] v(\bar y) \Big{)} 
\\[0.25cm] \geq (\mu^{1 + \gamma} - 1) c_f - a - m_f(|\bar x - \bar y|), 
\end{array}
\end{equation}
where $m_f$ is the modulus of continuity of $f$ in $\bar \Omega$, from which we derive 
$$
m_f(|\bar x - \bar y|) \to 0 \quad \mbox{as} \quad \epsilon \to 0.
$$ 
Then, fixing $a$ and $\epsilon$ small in order that the right-hand side of the last inequality is nonnegative, and canceling $|\bar p|$ we get that
\begin{equation}\label{esquema}
0 \leq \epsilon^{-2} O(\delta^{2 - 2s}) + L_{ij}[B_\delta^c] \tilde u(\bar x) - L_{ij}[B_\delta^c] v(\bar y).
\end{equation}

Now we see that
\begin{align*}
L_{ij}[B_\delta^c] \tilde u(\bar x) - L_{ij}[B_\delta^c] v(\bar y) = I_1 + I_2 + I_3 + I_4,
\end{align*}
where
\begin{align*}
I_1 & =  \int_{(\Omega - \bar x) \cap (\Omega - \bar y)} [\tilde u(\bar x + z) - v(\bar y + z) - (\tilde u(\bar x) - v(\bar y))]K_{ij}(z)dz, \\
I_2 & =  \int_{(\Omega - \bar x)^c \cap (\Omega - \bar y)} [\tilde u(\bar x + z) - v(\bar y + z) - (\tilde u(\bar x) - v(\bar y))]K_{ij}(z)dz, \\
I_3 & =  \int_{(\Omega - \bar x) \cap (\Omega - \bar y)^c} [\tilde u(\bar x + z) - v(\bar y + z) - (\tilde u(\bar x) - v(\bar y))]K_{ij}(z)dz, \\
I_4 & =  \int_{(\Omega - \bar x)^c \cap (\Omega - \bar y)^c} [\tilde u(\bar x + z) - v(\bar y + z) - (\tilde u(\bar x) - v(\bar y))]K_{ij}(z)dz.
\end{align*}
Using the maximality of $(\bar x, \bar y)$ in~\eqref{doubling}, we readily have 
$$
I_1 \leq 0.
$$
The estimates for $I_2$ and $I_3$ are similar, so we provide the details for the former. Using that $\Omega$ is bounded and that $u,v$ are locally bounded, there exists $R > 0$ large but independent of $\mu$ and $\epsilon$ such that
\begin{align*}
|I_2| \leq C (\| u\|_{L^\infty(B_R)} + \| v \|_{L^\infty(B_R)}) \int_{(\Omega - \bar x)^c \cap (\Omega - \bar y)} K_{ij}(z)dz,
\end{align*}
and using that $\bar x, \bar y$ are uniformly away from the boundary, there exists a constant $C_\rho > 0$ not depending on $\mu$ nor $\epsilon$ such that
\begin{equation}
\begin{array}{rcl}
|I_2| & \leq & C C_\rho (\| u\|_{L^\infty(B_R)} + \| v \|_{L^\infty(B_R)}) \int_{(\Omega - \bar x)^c \cap (\Omega - \bar y)} dz \\[0.3cm]
& = & C C_\rho (\| u\|_{L^\infty(B_R)} + \| v \|_{L^\infty(B_R)}) |(\Omega - \bar x)^c \cap (\Omega - \bar y)|,
\end{array}
\end{equation}
and so, since $|\bar x - \bar y| \to 0$ as $\epsilon \to 0$, we conclude that
$$
\max\{|I_2|, |I_3|\} = o_\epsilon(1),
$$
where $o_\epsilon(1) \to 0$ uniform on $\mu$. Finally, using that $u \in L^1_{2s}(\R^N)$, the fact that $\bar x, \bar y$ are away from the boundary, and that $u \leq v$ in $\mathbb{R}^N\setminus \Omega$, we arrive at
\begin{align*}
I_4 & \leq -M \int_{(\Omega - \bar x)^c \cap (\Omega - \bar y)^c} K_{i,j}(z)dz + C_\rho (1 - \mu) \\
& \leq -M \lambda \int_{(\Omega - \bar x)^c \cap (\Omega - \bar y)^c} |z|^{-(N + 2s)} dz + C_\rho (1 - \mu),
\end{align*}
where $C_\rho$ depends on $L^1_{2s}$ and $L^\infty_{loc}$ estimates of $u$, but not on $\epsilon$ nor $\mu$. Joining the above estimates into~\eqref{esquema}, we conclude that
\begin{align*}
0 \leq \epsilon^{-2} O(\delta^{2 - 2s}) - M \lambda \int_{(\Omega - \bar x)^c \cap (\Omega - \bar y)^c} |z|^{-(N + 2s)} dz + C_\rho (1 - \mu) + o_\epsilon(1),
\end{align*}
at this point, we take $\delta \to 0, \epsilon \to 0$ and $\mu \to 1$ to conclude
$$
0 \leq -\lambda M \int_{(\Omega - x_0)^c} |z|^{-(N + 2s)}dz,
$$
where $x_0 \in \Omega$ is such that $(u - v)(x_0) = M$, from which we arrive at a contradiction with the fact that $M> 0$. 
\end{proof}

%%%%%%%%
\subsection{Existence of viscosity solutions}\label{secinh}

Next, we prove the main result of this section. We sate existence of viscosity solutions for the Dirichlet problem \eqref{eq1} and $L^\infty$ estimates. Assuming, in addition, nondegeneracy condition for the source term $f$, we also derive uniqueness of solutions.  
 
\begin{theorem}\label{teo1}
Let $f \in C(\Omega) \cap L^\infty(\Omega)$ and  $g \in C(\mathbb{R}^N\setminus \Omega) \cap L_{2s}^1(\mathbb{R}^N\setminus \Omega)$. The following statements hold:

\medskip

\begin{enumerate}
\item[(i)] (Existence) There exists viscosity solution $u \in C(\R^N)$ for problem~\eqref{eq1}.

\medskip

\item[(ii)] ($L^\infty$-bounds) Given $u \in C(\R^N)$ the solution to \eqref{eq1}, there holds
$$
\| u \|_{L^\infty(\Omega)} \leq C (1 +  \| f \|_\infty + \| g\|_{L^\infty(\partial \Omega)} + \| g \|_{L^1_{2s}(\mathbb{R}^N\setminus \Omega)}),
$$
for some $C>0$ depending only on $N, s, \gamma$ and $\Omega$.

\medskip

\item[(iii)] (Uniqueness) Assuming $\inf_{\Omega} |f|  > 0$, problem ~\eqref{eq1} has unique solution.

\end{enumerate}
\end{theorem}

\begin{lemma}\label{barrera}
For $M > 0$, denote $\varphi_M(x) = (M^2 - |x|^2)_+$. Then, there exists $M \geq 2$ large enough just depending on $N, s$ and the ellipticity constants, and $c_M > 0$ such that  
\begin{equation*}
\mathcal M^+ \varphi_M(x) \leq -c_M \quad \mbox{in} \ B_1.
\end{equation*}
\end{lemma}

\begin{proof}
This is a well-known result, but we provide the proof for completeness.
For each $K$ we have
\begin{equation}
\begin{array}{rcr}
L_K \varphi_M(x) & = & \displaystyle \mathrm{P.V.} \int_{B_{M - |x|}(x)} [\varphi_M(y) - \varphi_M(x)] K(x - y)dy \\[0.4cm]
& & + \displaystyle \int_{\mathbb{R}^N \setminus B_{M - |x|}(x)} [\varphi_M(y) - \varphi_M(x)] K(x - y)dy \\[0.4cm]
& \leq & \displaystyle\mathrm{P.V.} \int_{B_{M - |x|}(x)} [\varphi_M(y) - \varphi_M(x)] K(x - y)dy \\[0.4cm]
& = & \displaystyle \frac{1}{2} \int_{B_{M - |x|}(0)} \int_{0}^{1} \langle D^2 \varphi_M(x + tz) z, z \rangle dt K(z)dz,
\end{array}
\end{equation}
and since the function $\varphi_M$ is concave in $B_{M - |x|}(x)$ we conclude the claim. 

\end{proof}

Now we are in a position to provide the main result of this section.

\begin{proof}[Proof of Theorem~\ref{teo1}]
We focus our arguments for proving $(i)$. For each $\epsilon > 0$, we consider the auxiliary problem
\begin{equation}\label{eqeps}
\left \{
\begin{array}{rll} (\epsilon + |Du|^{\gamma}) I u & = f \quad & \mbox{in} \ \Omega, \\ u & = g \quad & \mbox{in} \ \mathbb{R}^N\setminus \Omega, \end{array} \right .
\end{equation}
Since the problem above is uniformly elliptic, it has a unique solution $u^\epsilon \in C(\bar \Omega) \cap C^{1, \alpha}(\Omega)$. For the reader's convenience, we provide a sketch of the proof in Appendix \ref{existenceapx}.  

Next, consider $x_0 \in \mathbb{R}^N\setminus \Omega$ and $M \gg 1$ such that $\mathrm{dist}(x_0, \Omega) = 1$ and $\Omega \subset B_{M/2}(x_0)$. By translation invariance, we may assume that $x_0 = 0$. Denote $d_\Omega = \mathrm{diam}(\Omega) + 1$. 
From Lemma~\ref{barrera}, and scaling properties of the extremal operator, we notice that function
$$
\varphi(x) := \varphi_M(x/d_\Omega)
$$
satisfy
$$
\mathcal M^+ \varphi \leq -c_M d_\Omega^{-2s} \quad \mbox{in } \; \Omega.
$$
Next, for $A > 1$ to be fixed, consider function 
$$
V(x) = A \varphi(x) + \chi_{\mathbb{R}^N \setminus B_{2M/3}}(x) g(x), \quad x \in \R^N.
$$
A simple computation shows that for each $x \in \Omega$ we have
$$
\mathcal M^+ V(x) \leq -A c_M d_\Omega^{-2s} + C\Lambda \| g \|_{L^1_{2s}(B_{2M/3})}, 
$$
for some $C > 0$ just depending on $N, s$. Then, choosing $A$ large enough in terms of $M$ and $g$, we guarantee that $V$ satisfies
$$
\MM^+ V(x) \leq -1 \quad \mbox{in } \; \Omega,
$$
and $V \geq g$ in $\mathbb{R}^N\setminus \Omega$.

On the other hand, there exists $\tilde c_M > 0$ depending on $M, N$ such that $|DV(x)| \geq A\tilde c_M$ for all $x \in \Omega$. Then, enlarging $A$ in terms of $\gamma, \| f \|_\infty$, but not on $\epsilon$, we conclude that 
\begin{equation}\label{Lisa}
(\epsilon + |DV|^\gamma) IV \leq (\epsilon + |DV|^\gamma) \MM^+ V \leq -(1 + \| f\|_\infty) \quad \mbox{in} \ \Omega.
\end{equation}

For $r > 0$, we denote $\Omega_r = \{ x \in \Omega : d(x) < r \}$, and $\Omega^r = \Omega + B_r(0)$. For $\rho > 0$ small enough, we consider $g_\rho \in C(\R^N)$ such that, for all $\rho > 0$, $g_\rho \in C^2( \Omega^1)$, $g_\rho = g$ in $(\Omega^2)^c$ and such that $\| g_\rho - g\|_{L^\infty(\Omega^2)} \leq \rho$. Hence, we can assume that for each $\rho > 0$, we have that
\begin{equation}
\begin{array}{ccl}
\| g_\rho \|_{L^\infty(\Omega^2)} & \leq & \| g \|_{L^\infty(\Omega^2 \setminus \Omega)} + 1, \\[0.2cm]
 \| Dg_\rho \|_{L^\infty(\Omega^2)} & \leq & C \rho^{-1}, \\[0.2cm] 
 \| D^2 g_\rho \|_{L^\infty(\Omega^2)} & \leq & C \rho^{-2},
\end{array}
\end{equation}
for some constant $C$ just depending on $N, \Omega$.

From \cite[Lemma 3.1]{DQT1}, there exists $\delta > 0$, $\beta_0 \in (0,s)$ and $c_0 > 0$ such that 
\begin{equation*}
\mathcal M^+ d_+^\beta(x) \leq -c_0 d^{\beta - 2s}(x), \quad \mbox{for} \ x \in \Omega_\delta,
\end{equation*} 
for each $\beta \in (0, \beta_0)$. Now, taking $C_1 > 1$ to be fixed, we consider
$$
\psi_\rho(x) = g_\rho(x) + \rho +  C_1 d_+^\beta(x),
$$
which satisfies
$$
(\epsilon + |D\psi_\rho|^\gamma) I \psi_\rho \leq (\epsilon + |\beta C_1 d^{\beta - 1}(x) Dd(x) + Dg_\rho(x)|)^\gamma (-c_0 C_1 d^{\beta - 2s}(x) + C_g),
$$
for each $x \in \Omega_\delta$. Notice that 
$$
C_g \leq C \Lambda (\rho^{-2} + \| g \|_{L^\infty(\tilde \Omega)} + \| g \|_{L^1_{2s}(\mathbb{R}^N\setminus \Omega)}),
$$
for some universal constant $C > 0$.
In view of this,
select $C_1 = C_0(\rho^{-2} + \| f \|_\infty + 1)$ for $C_0$ large enough, and independent of $\epsilon, \rho$ and $f$, such that
\begin{equation*}
(\epsilon + |D\psi_\rho|^\gamma) I \psi_\rho \leq -(\| f \|_\infty + 1) \quad \mbox{in} \ \Omega_\delta,
\end{equation*}
for all $\rho > 0$ and all $\epsilon > 0$. By standard arguments, we have that function
\begin{equation*}
\psi(x) := \inf_{\rho > 0} \{ \psi_\rho(x) \},
\end{equation*}
satisfies 
\begin{equation*}
(\epsilon^2 + |D\psi|^2)^{\gamma/2} I \psi \leq -(\| f \|_\infty + 1) \quad \mbox{in} \ \Omega_\delta,
\end{equation*}
for each $\epsilon > 0$.

Enlarging $C_0$ in terms of $V$, but not on $\rho$, we can assume that $\psi \geq V$ on $\Omega \setminus \Omega_\delta$. Hence, we conclude that  
\begin{equation}\label{barrierV}
\overline V := \min \{ V, \psi \}
\end{equation}
is a viscosity supersolution to~\eqref{eqeps}. By comparison principle, we get $u^\epsilon \leq \overline V$ in $\Omega$. A lower bound can be found similarly, from which the family $\{ u^\epsilon \}_\epsilon$ is uniformly bounded in $\Omega$. We now apply arguments in~\cite{DPT} which guarantee interior H\"older estimates for $\{ u^\epsilon \}_\epsilon$. Then, letting $\epsilon \to 0$, we conclude the existence of a solution $u \in C(\R^N)$ for ~\eqref{eq1}. 

Consequently, $u$ attains the boundary condition continuously. Hence, for each $x \in \Omega$ near the boundary, we have $\bar V(x) = \psi(x)$. Then, if we denote $\hat x$ the projection of $x$ to $\partial \Omega$, we see that for each $\rho > 0$ we have 
\begin{equation}
\begin{array}{ccl}
u(x) - g(\hat x) & \leq & \psi(x) - g(\hat x) \\[0.2cm] 
& \leq & C\rho^{-2} d^\beta(x) + \rho + g_\rho(x) - g(\hat x) \\[0.2cm]
& \leq & C\rho^{-2} d^\beta(x) + 2\rho + m_g(|x - \hat x|) \\[0.2cm]
& \leq & C\rho^{-2} d^\beta(x) + 2\rho + m_g(d(x)),
\end{array}
\end{equation}
for some modulus of continuity $m_g$ which depends only on $g$ and the smoothness of boundary $\partial \Omega$. Minimizing on $\rho$, we obtain existence of constant $C > 0$, such that
\begin{align*}
u(x) - g(\hat x) \leq C d^{\beta/2}(x) + m_g(d(x)).
\end{align*}

A similar lower bound can be obtained, from which we conclude that 
$$
|u(x) - g(\hat x)| \leq \tilde m(d(x)),
$$
for some modulus of continuity $\tilde m$,
from which we conclude that for each $x_0 \in \partial \Omega$, one holds $\psi(x) \to g(x_0)$ as $x \to x_0$, for $x \in \Omega$.

A priori estimates in $(ii)$ follow by comparison principle and the (strict) barriers constructed in~\eqref{Lisa}. Uniqueness property $(iii)$ follows by comparison principle obtained in Proposition~\ref{comparison}.
\end{proof}

%%%%%%%%
\section{Degenerate $s$-harmonic functions}\label{homsec}

In this section, we prove a more general version of Theorem~\ref{teomaxmin}. For this, we consider the class of nonlocal operators $I$ with the form~\eqref{operator} and discuss multiplicity properties for degenerate $I$-harmonic functions, which are solutions for the problem 
\begin{equation}\label{Dirichlet0geral}
\left\{ 
\begin{array}{rcl} |Du|\, I u = 0  & \mbox{in} & \Omega \\[0.2cm]
u = g & \mbox{in} & \mathbb{R}^N\setminus \Omega. 
\end{array} 
\right.
\end{equation}
We consider the problem above, since solutions of $|Du|^\gamma I u = 0$, for some $\gamma>0$, also solves in the viscosity sense $|Du|\, I u = 0$. In addition, we highlight the unique solution of problem
\begin{equation}\label{eqharmonicgeral}
\left\{ 
\begin{array}{rcl} 
I u = 0  & \mbox{in} & \Omega, \\[0.2cm]
u = g & \mbox{in} & \mathbb{R}^N\setminus \Omega. 
\end{array} 
\right. 
\end{equation} 
solves problem \eqref{Dirichlet0geral}. According notation for case $I=\Delta^s$, we denote 
$$
\mathcal{H}(s, g, \Omega) :=\{u: \mathbb{R}^N \to \mathbb{R} \; | \; u \mbox{ solves } \eqref{Dirichlet0geral}\}.
$$
We also define 
\begin{equation}\label{RIO}
\mathcal{P}(\Omega):=\{g: \mathbb{R}^N \setminus \Omega \to \mathbb{R} \; | \; \mbox{problem } \eqref{Dirichlet0geral} \mbox{ has unique solution} \}.
\end{equation}

%%%%%%
\begin{theorem}\label{teomaxmingeral}
Let $s \in (0,1)$ and $\gamma > 0$. Assume $\Omega \subset \R^N$ a bounded domain with $C^2$ boundary, $g \in L^1_{2s}(\mathbb{R}^N\setminus \Omega) \cap C(\mathbb{R}^N\setminus \Omega)$, and $w$ the unique solution of problem \eqref{eqharmonicgeral}. The following statements hold:

\medskip

\begin{enumerate} 
\item[(i)] There exist $\bar u, \underline u \in \mathcal{H}(s,g,\Omega)$, such that for each $u \in \mathcal{H}(s,g,\Omega)$, 
$$
\underline u \leq u \leq \bar u \quad \mbox{in} \ \Omega.
$$

\medskip

\item[(ii)] Assume $\overline u$ (resp. $\underline u$) touches $w$ at some point in $\Omega$, then 
$$
\overline u = w \quad (\mbox{resp. } \underline u = w).
$$

\item[(iii)] Assume $\overline u$ touches $\underline u$ at some point in $x \in \Omega$, then 
$$
g \in \mathcal{P}(\Omega).
$$

\medskip

\item[(iv)] Assume $|Dw|\neq 0$ in $\Omega$, then
$$
g \in \mathcal{P}(\Omega).
$$

\medskip
\item[(v)] Assume that $w$ satisfies the following property: for each sequence $\{ u_k\}_k \subset C(\Omega)$ such that $u_k \to w$ locally uniformly in $\Omega$, $u_k$ has a local maximum point (resp. local minimum point) $x_k \in \Omega$ with uniform distance to the boundary. Then $\underline u < w$ (resp. $\bar u > w$) in $\Omega$.
\end{enumerate}
\end{theorem}
%%%%

We start introducing approximating solutions: let $u^\eta$ (resp. $u_\eta$) be the unique viscosity solution to the problem
\begin{equation}\label{eqeta}
\left \{ \begin{array}{rll} |Du| I u & = -\eta \ (resp. = +\eta) \quad & \mbox{in} \ \Omega, \\
u & = g \quad & \mbox{in} \ \mathbb{R}^N\setminus \Omega. \end{array} \right .
\end{equation}
By comparison principle, we have $u_\eta \leq u^\eta$ in $\Omega$, and both functions are locally Lipschitz continuous in $\Omega$. Moreover, 
there exists a constant $C > 0$ just depending on $N, s, \Omega$ and the ellipticity constants, but not on $\eta$, such that
$$
\max\{\| u_\eta \|_{L^\infty(\Omega)}, \ \| u^\eta \|_{L^\infty(\Omega)}\} \leq C (1 + \| g\|_{L^\infty(\partial \Omega)} + \| g \|_{L^1_{2s}(\mathbb{R}^N\setminus \Omega)}).
$$

\begin{proof}[Proof of (i)] Let $\eta_0 > 0$ and for each $\eta \in (0,\eta_0]$ let $u_\eta$ be the unique solution to~\eqref{eqeta}. By comparison principle we have $-\| g \|_\infty \leq u_\eta \leq u_{\eta_0}$. Thus, from Theorem \ref{PTreg} the family $\{ u_\eta \}_\eta$ is uniformly bounded and equicontinuous. Hence, by Arzela-Ascoli's theorem and stability results of viscosity solutions, taking the uniform limit $\bar u = \limsup_{\eta \to 0} u_\eta$ we obtain the existence of a solution to~\eqref{Dirichlet0geral}. The maximality of $\bar u$ comes with the fact that any solution $u$ to~\eqref{Dirichlet0geral} is a subsolution to the problem solved by $u_\eta$. Then, by comparison principle, we have $u \leq u_\eta$ and the result holds by taking limit. A similar analysis can be done for the minimal solutions.
\end{proof}

\begin{proof}[Proof of (ii)] Assume that $\underline u(x_0) = \bar u(x_0)$ for some point $x_0 \in \Omega$. Particularly, $\bar u(x_0) = w(x_0)$, where $w$ is the unique solution to~\eqref{eqharmonicgeral}. Let $S = \{ x \in \Omega : \bar u(x) = w(x) \}$. Then, $\mathbb{R}^N \setminus S$ is relatively open. If $\mathbb{R}^N \setminus  S \neq \emptyset$, then there exists $a > 0$ and a point $y \in \Omega$ such that $w + a \leq \bar u \leq u^\eta$ in $B_a(y)$ for all $\eta$. Notice that $u^\eta$ is such that $I u^\eta \leq 0$ in $\Omega$, since no smooth function touching $u^\eta$ from below has null gradient. Then, consider $O = B_{a/4}(x_0)$, and notice that there exists $y' \in \Omega$ such that $B_{a/4}(y') \subset B_a(y)$ and $B_{a/4}(y') \subset \mathbb{R}^N \setminus O$.

Then, consider the function 
$$
W = w + \varphi_\epsilon + a \mathbf{1}_{B_{a/4}(y')},
$$
where for $\epsilon \in (0,a)$, $\varphi_\epsilon$ is a nonnegative, smooth function with support in $B_\epsilon(x_0)$, such that $\varphi_\epsilon(x_0) = \epsilon$, and $\| \varphi_\epsilon \|_{C^2(\R^N)} \to 0$ as $\epsilon \to 0$. Notice that $W \leq u^\eta$ in $\mathbb{R}^N \setminus O$ and that for each $x \in O$, we have
$$
I W(x) \geq Iw(x) + \MM^- \varphi_\epsilon(x) + a \MM^- \mathbf{1}_{B_{a/r}(y')} (x) \geq o_\epsilon(1) + c_a, 
$$
where $c_a > 0$ and $o_\epsilon(1) \to 0$ as $\epsilon \to 0$. Then, we fix $\epsilon > 0$ small enough in terms of $a$ to conclude that 
$$
IW \geq 0 \quad \mbox{in } \; O. 
$$
Then, by comparison principle, we have 
$$
W(x_0)=w(x_0) + \epsilon \leq u^\eta(x_0)
$$
but this contradicts the fact that $u^\eta(x_0) \to \bar u(x_0) = w(x_0)$ as $\eta \to 0$. 
\end{proof}

\begin{proof}[Proof of (iii)]
One easily follows from (i) and (ii).
\end{proof}

\begin{proof}[Proof of (iv)]
Assume $\inf_{\Omega} |Dw| = c_0 > 0$. For $C > 0$ to be fixed, consider 
$$
w_\eta = w + C\eta \mathbf{1}_\Omega.
$$
Hence, we see that for each $x \in \Omega$, we obtain
\begin{equation}\nonumber
\begin{array}{ccl}
|Dw_\eta(x)|^\gamma I w_\eta & \leq & \displaystyle |Dw(x)|^\gamma (I w(x) - C \eta \inf_{K} \int_{\mathbb{R}^N\setminus \Omega} K(x- y)dy) \\
& \leq & - C c_0 \displaystyle \lambda \eta \int_{\mathbb{R}^N\setminus \Omega} |y - \hat x|^{-(N + 2s)}dy,
\end{array}
\end{equation}
where $\hat x$ is a point in $\Omega$ maximizing the distance to the boundary. Hence, there exists constant $C$ depending on $c_0$, $\mathrm{diam}(\Omega)$, $\lambda$, $N$ and $s$, such that 
$$
|Dw_\eta|^\gamma I w_\eta \leq -\eta \quad \mbox{in } \; \Omega.
$$
Then, by comparison principle, we have $w \leq u^\eta \leq w_\eta$ for all $\eta$, and taking the limit as $\eta \to 0$ we conclude that $\bar u = w$ in $\Omega$. A similar analysis can be done for $\underline u$ to conclude that $\underline u = w$.
\end{proof}

\begin{proof}[Proof of (v)] We prove the result for the maximal solution, the proof for the minimal solution follows the same lines. Assume by contradiction $\underline u = w$. Then, $u_\eta \to w$ locally uniformly in $\Omega$ as $\eta \to 0^+$, from which we have $u_\eta$ has a local maximum point in $x_\eta \in \Omega$ with uniform distance to the boundary. Then, a constant function touches from above $u_\eta$ at $x_\eta$. Using the viscosity inequality for $u_\eta$ at $x_\eta$, we have
$
0 \geq \eta > 0,
$
which is a contradiction.
\end{proof}

\begin{remark}%\label{rmkiv}
We stress the fact that $\underline u, \overline u$ are understood as minimal and maximal solutions of the problem in the class of bounded solutions, since there are plenty of large solutions (blowing-up on the boundary) for problem~\eqref{eqharmonicgeral}, see for instance~\cite{DQThar}. 
\end{remark}

We conclude with the following stability result for the extremal solutions. 

\begin{proposition}\label{propestmax}
Let $g_k, g \in C(\mathbb{R}^N\setminus \Omega)$ such that $g_k \to g$ locally uniform in $\mathbb{R}^N\setminus \Omega$. In addition, for some $M > 0$ and $ \sigma \in (0,2s)$, assume that 
\begin{equation}\label{growth of g}
\max\{|g_k|, |g|\} \leq M(1 + |x|)^\sigma, \quad \mbox{for } \; x \in \mathbb{R}^N\setminus \Omega.
\end{equation}
Denote $\bar u_k, \bar u$ the maximal solutions associated to $g_k$ and $ g$, respectively. Then, $\bar u_k \to \bar u$ uniformly in $\Omega$. The same result holds for the associated minimal solutions.
\end{proposition}

\begin{proof}
By Theorem~\ref{teo1}, we have $\limsup_{k \to \infty} \bar u_k$ exists in the uniform sense in $\Omega$ and by stability this function solves $|Du|^\gamma Iu = 0$ in $\Omega$, with exterior data $g$. Thus, we have 
$$
\limsup_{k \to \infty} \bar u_k \leq \bar u
$$ 
by maximality of $\bar u$. 

Let $\epsilon > 0$. By the assumption (\ref{growth of g}) there exists $R_0 > 1$ large enough depending on $\epsilon$ such that $\Omega \subset B_{R_0/2}$ and for all $R > R_0$ we have
$$
\int_{\mathbb{R}^N \setminus B_R} |g_k(y) - g(y)| K(x - y)dy \leq \epsilon,
$$
for all $x \in \Omega$ and all $K \in \mathcal L_0$. Fixed $R$ in this way, for all $k$ large enough we have
$$
g_k + \epsilon \geq g \quad \mbox{in} \ \mathbb{R}^N \setminus B_R.
$$ 

Now, for $\eta > 0$, denote $u_k^\eta$ the approximating function to $\bar u_k$ and  for $C > 1$ to be fixed, we consider $v: \R^N \to \R$ given by
$$
v(x) = \left \{ \begin{array}{ll} u_k^\eta + C\epsilon \quad & \mbox{in} \ \Omega, \\ g \quad & \mbox{in} \ \mathbb{R}^N\setminus \Omega. \end{array} \right .
$$

Notice this function is lower semicontinuous in $\R^N$, and we have $v \geq g$ in $\mathbb{R}^N\setminus \Omega$. Given $x \in \Omega$ and $K \in \mathcal L_0$, we have
\begin{align*}
L_K v(x) = & \int_{\Omega}[u^\eta_k(y) - u^\eta_k(x)]K(x - y)dy \\
& + \int_{\mathbb{R}^N\setminus \Omega} [g(y) - C\epsilon - u_k^\eta(x)]K(x - y)dy \\
\leq & L_K u^\eta_k(x) - (C-1)\epsilon \int_{B_R \setminus \Omega} K(x - y) dy \\
& + \int_{\mathbb{R}^N \setminus B_R} |g(y) - g_k(y)| K(x - y)dy.
\end{align*}

By the boundedness of $\Omega$, there exists $c_0 > 0$ such that, for all $R \geq R_0$ large enough, we have 
$$
\int_{B_R\setminus \Omega} K(x-y)dy \geq c_0, 
$$
for all $x \in \Omega$ and all $K \in \mathcal L_0$. Thus, taking $C$ large enough, we conclude that
$$
L_K v(x) \leq L_K u^\eta_k(x), \quad \mbox{for all} \ x \in \Omega, \ K \in \mathcal L_0, \ \eta > 0,
$$
and since $Dv(x) = Du^\eta_k(x)$ (in the viscosity sense), we conclude that $v$ satisfies
$$
|D v|^\gamma I v(x) \leq -\eta \quad \mbox{in} \ \Omega,
$$
and is such that $v \geq g$ in $\mathbb{R}^N\setminus \Omega$. Thus, if we denote $u^\eta$ as the approximating function to $\bar u$, we conclude that
$$
u^\eta \leq u_k^\eta + C\epsilon \quad \mbox{in} \ \Omega,
$$
and taking limits as $\eta \to 0$ we arrive at
$$
\bar u \leq \bar u_k + C\epsilon \quad \mbox{in} \ \Omega, 
$$
and taking $\liminf_{k \to \infty}$ and $\epsilon \to 0$ next, we arrive at
$$
\bar u \leq \liminf_{k \to \infty} \bar u_k \leq \limsup_{k \to \infty} \bar u_k \leq \bar u,
$$
and the result follows.
\end{proof}

%%%%%%%%%
\section{Multiplicity in non-local scenarios} \label{sec-examples}

Now we present some examples of uniqueness and non-uniqueness for the homogeneous case of the Dirichlet problem~\eqref{eq}. We restrict ourselves to the case $I = \Delta^s$ in order to take advantage of the larger information about $s$-harmonic functions available in the literature. One of the main tools is the following version of the strong maximum principle for $(s, \gamma)$-subharmonic functions. We postpone the proof of the result below to the final of the section. 

%%%
\begin{proposition}\label{strong maximum principle}
Assume $\Omega$ is connected. Let $u \in USC(\Omega) \cap L^\infty(\R^N)$ be a viscosity subsolution to $|Du|^\gamma \Delta^s u \geq 0$ in $\Omega$ and assume there exists $x_0 \in \Omega$ such that $\max_{\R^N} u = u(x_0)$. Then, $u$ is constant in $\Omega$.
\end{proposition}
\noindent
\textit{Uniqueness for linear exterior datum.}
In case $g(x)=a\cdot x + b$ for $a \in \R^N\setminus \{0\}$ and $b \in \R$, we consider $g \cdot \mathbf{1}_{\R^N\setminus \Omega}$ as the exterior data. We observe that $g\cdot  \mathbf{1}_{\Omega}$ is the corresponding $s$-harmonic function. Since $|Dg|\neq 0$ in $\Omega$, we apply Theorem~\ref{teomaxmin}, which implies $g \in \mathcal{P}(\Omega)$. The case $g(x)=b$, leads to $ \mathbf{1}_{\Omega}$ as the unique $s$-harmonic function. By Proposition \ref{strong maximum principle}, we have $\mathcal{H}(\gamma,s,g,\Omega)=\{g \cdot  \mathbf{1}_{\Omega}\}$, and so $g \in \mathcal{P}(\Omega)$. 

\vspace{0.2cm}

\noindent
\textit{Constructing a less explicit example.} By following results in~\cite{DSV}: let $f \in C^1(\Omega)$ such that $\inf_{x \in \Omega} |Df(x)| = c_0 > 0$. Then, there exists $R > 1$ large enough (in such a way $\Omega \subset B_{R/2}$) and $w$ a $s$-harmonic function in $\Omega$ such that $w = 0$ in $\mathbb{R}^N\setminus B_R$, and $\| w - f \|_{C^1(\Omega)} \leq c_0/2$. Thus, we have that $|Dw|$ does not vanish in $\Omega$ and so, from Theorem~\ref{teomaxmin}, $g \in \mathcal{P}(\Omega)$, for $g=f\cdot \mathbf{1}_{B_R\setminus \Omega}$ defined in $\R^N \setminus \Omega$. 
\vspace{0.1cm}

\noindent
\textit{Existence of, at least, two $(\gamma, s)$-harmonic functions.} Let $\Omega = B_1$ and let $g \in C(\R^N)$ with $g \geq 0$ in $\R^N$, $g = 0$ in $B_1$ and $g=1$ in $\mathbb{R}^N\setminus B_2$. Note that $g \in \mathcal{H}(\gamma,s,g,B_1)$ and $\Delta^s g>0$ in $B_1$. In this case, the unique $s$-harmonic function $w \in \mathcal{H}(\gamma,s,g,B_1)$ is positive in $B_1$. On the other hand, from Theorem \ref{teomaxmin} and the strong maximum principle, the minimal solution $\underline u$ cannot have a negative minimum in $\Omega$, and so $\underline u = g$.
\vspace{0.1cm}

\noindent
\textit{Maximal solution coinciding with the $s$-harmonic function.}
For $2s > 1$, assume $\Omega = B_1$, $g: \mathbb{R}^N\setminus B_1 \to \R$ radial, and that the unique $s$-harmonic function $w \in \mathcal{H}(\gamma,s,g,B_1)$ is radially non-increasing. Then, for the maximal solution $\bar u$, we have $\bar u = w$. In fact, let $\sigma := 1 + \frac{2s - 1}{1 + \gamma} < 2s$ and for $R > 2$, consider the function 
$$
\xi(x) = (R^\sigma - |x|^\sigma)_+, \quad x \in \R^N.
$$
For $0 < |x| < 1$ and $R > 2$, we have
\begin{align*}
\Delta^s \xi(x) & = -\Delta^s |x|^\sigma + C_{N, s}\int_{\mathbb{R}^N\setminus B_R} \frac{|y|^\sigma- R^\sigma}{|x - y|^{N + 2s}}dy \\
& \leq -c |x|^{\sigma - 2s} + 2^{N + 2s} C_{N, s}\int_{\mathbb{R}^N \setminus B_R} \frac{|y|^\sigma}{|y|^{N + 2s}}dy \\
& = -c |x|^{\sigma - 2s} + \frac{2^{N + 2s} R^{\sigma - 2s} C_{N, s}}{2s - \sigma},
\end{align*}
where $c > 0$ in view of the estimates in~\cite{FQ} (see lemmas 3.1, 3.2 and 3.3 there). Then, by choosing $R$ large enough, we get that
$$
\Delta^s \xi(x) \leq -\frac{c}{2} |x|^{\sigma - 2s}, \quad \mbox{for} \ |x| < 1,
$$
in the viscosity sense. Now, for $\epsilon > 0$ consider the function
$$
V(x) = w(x) + \epsilon \xi(x), \quad x \in \R^N.
$$

For each $0 < |x| < 1$, we have that
$$
|DV(x)|^\gamma \Delta^s V(x) = \epsilon |Dw(x) - \epsilon \sigma |x|^{\sigma - 1} \hat x|^\gamma \Delta^s \xi(x), 
$$
and since $w \in C^{2s + \alpha}(B_1)$ is radially decreasing, we have
$$
|DV(x)| \geq \epsilon \sigma |x|^{\sigma - 1},
$$
from which
$$
|DV(x)|^\gamma \Delta^s V(x) \leq -\epsilon^{1 + \gamma} \sigma^\gamma c/2, \quad \mbox{for} \ x \in B_1,
$$
in the viscosity sense. Then, since $V \geq g$ in $\mathbb{R}^N\setminus B_1$, we can take $\epsilon = \epsilon(\eta) \to 0$ as $\eta \to 0$ such that $V$ is a supersolution to $u^\eta$ solving~\eqref{eqeta} (with right-hand side $-\eta$). Then, by Proposition~\ref{comparison}, we conclude that
$$
u^\eta \leq w + \epsilon(\eta) \xi \quad \in B_1.
$$
Taking $\eta \to 0$, we conclude that $\bar u \leq w \leq \bar u$, and the maximal solution meets the $s$-harmonic function. 

We do not know under which conditions the $s$-harmonic function satisfies the monotonicity condition requested in this example. In Figure~\ref{fig:Kariri}, we provide an example of such a function.

\medskip

\noindent
\textit{Existence of, at least, four $(\gamma,s)$-harmonic functions.} Consider $\Omega = (-1,1)$ and an odd, bounded, smooth function $g: \R \to \R$ such that $g = 0$ in $\Omega$ and not identically zero in $\R$. In particular, we assume $g > 0$ in $(1, +\infty)$.

We have that the unique solution $w$ to $\Delta^s w = 0$ in $\Omega$ with $w = g$ in $\mathbb{R}^N\setminus \Omega$ is not identically zero in $\Omega$. If this is the case, for $x \in \Omega$ close to $1$ we have
\begin{align*}
0 = & C_{1,s} \int_{-\infty}^{-1} g(y) |x - y|^{-(1 + 2s)}dy + C_{1,s}\int_{1}^{+\infty} g(y) |x - y|^{-(1 + 2s)}dy \\
  = & -C_{1,s}\int_{1}^{+\infty} g(y) |x + y|^{-(1 + 2s)} dy + C_{1,s}\int_{1}^{+\infty} g(y) |x - y|^{-(1 + 2s)}dy \\
  = & C_{1,s}\int_{1}^{+\infty} g(y) \Big{(} |x - y|^{-(1 + 2s)} - |x + y|^{-(1 + 2s)} \Big{)}dy > 0,
\end{align*}
which is a contradiction. In addition, we have $w$ is odd, since by defining $\tilde w(x) = -w(-x)$, we have $\tilde w$ satisfies the same equation as $w$. Thus, $w$ changes sign in $\Omega$. This function $w$ is a solution for the degenerate equation.

On the other hand, the function $g$ is itself a solution to the degenerate equation and this is different from $w$. Moreover, the maximal and minimal solutions $\bar u$ and $\underline u$ cannot be equal to $w$ or $g$ by maximality/minimality. Then, we have four $(\gamma, s)-$harmonic functions with the same exterior data.

The question here is how do we find the trivial solution equal to zero in $\Omega$? This involves a more general question regarding the existence of solutions aside the $s$-harmonic, and the extremal ones. In fact, this suggests different ways to obtain solutions to the homogeneous Dirichlet problem. In this example, we notice that for each $\epsilon \in (0,1)$, the function $u_0: \R \to \R$ given by $u_0 = 0$ in $\Omega$, $u = g$ in $\mathbb{R}^N\setminus \Omega$ is a classical solution to the problem $(\epsilon^2 + |u_x|^2)^{\gamma/2} \Delta^s u = -\epsilon^{\gamma} f$ in $(-1,1)$, where $f(x) = \Delta^s (\chi_{\mathbb{R}\setminus (-1,1)} g)(x)$. Notice that $f$ changes sign in $(-1,1)$.

\medskip

We conclude this section with the
\begin{proof}[Proof of Proposition~\ref{strong maximum principle}]
Assume $N \geq 2$. For $\sigma \in (2s - N, 0)$, denote $E(x) = |x|^\sigma, x \neq 0$. We have the existence of $c < 0$ such that
$$
\Delta^s E(x) = c |x|^{\sigma - 2s}, \quad x \neq 0,
$$
see \cite[Lemmas 3.1, 3.2, and 3.3]{FQ}. For $\epsilon > 0$, let $E_\epsilon \in L^\infty(\R^N)$ given by
$$
E_\epsilon(x) = \min \{ E(x), \epsilon^{-\sigma} \}, \ x \in \R^N. 
$$

We claim that for each $r > 0$, there exists $\epsilon$ small enough in terms of $r$ such that 
$$
\Delta^s E_\epsilon(x) < 0 \quad \mbox{for} \ x \in B_r \setminus B_{r/2}.
$$

In fact, for $\epsilon \leq r/2$ and $r/2 \leq |x| \leq r$ we have
\begin{align*}
\Delta^s E_\epsilon(x) = & \Delta^s E(x) + C_{N,s} \int_{B_\epsilon} \frac{\epsilon^\sigma - |y|^\sigma}{|x - y|^{N + 2s}}dy \\
\leq & c |x|^{\sigma - 2s} + C_{N,s} \epsilon^\sigma \int_{B_\epsilon} \frac{dy}{|x - y|^{N + 2s}} \\
\leq & c r^{\sigma - 2s} + C_{N,s} \epsilon^\sigma 2^{N + 2s} r^{-(N + 2s)} |B_\epsilon|,
\end{align*}
where $|B_\epsilon|$ denotes the Lebesgue measure of the ball $B_\epsilon$.

Since $2s - N < \sigma$, we can take $\epsilon$ small enough in terms of $r$ such that
$$
\Delta^s E_\epsilon(x) \leq \frac{c}{2} r^{\sigma - 2s} \quad \mbox{for } \ x \in B_r \setminus B_{r/2},
$$
and the claim follows.

This function will play the role of a barrier in the strong maximum principle.

Denote 
$
C = \{ x \in \Omega : u(x) = u(x_0) \},$ which is nonempty and closed (in the induced topology of $\Omega$). Then, $\Omega \setminus C$ is open. By contradiction, we assume it is nonempty. Then, there exists $x_1 \in \Omega$ and $r > 0$ such that $B_r(x_1) \subset \Omega \setminus C$. By enlarging $r$ we can assume that $\partial B_r(x_1) \cap \partial C \neq \emptyset$. Let $x_2$ be in the intersection. We can take $r$ arbitrarily small by moving $x_1$ in the direction of $x_2$. For simplicity, we assume that $x_1 = 0$.
    
Let $v(x) = -(E_\epsilon(x) - E_\epsilon(x_2))$. Notice that $v$ is bounded, and $v(x) \geq 0$ for $|x| \geq r$. For $\beta > 0$ small enough in terms of $r$ (and $\epsilon$), we have the function
$$
x \mapsto u(x) - \beta v(x)
$$
attains a global maximum at a point $\tilde x \in \overline B_r \setminus B_{r/2}$. In fact, we have $u - \beta v < M$ in $\mathbb{R}^N \setminus \overline B_r$. Since $u < M$ in $B_{r/2}$, taking $\beta$ small enough we have $u - \beta v < M$ there. Finally, since $u(x_2) - \beta v(x_2) = M$, we conclude the existence of $\tilde x$. Then, by using $\beta v$ as a test function for $u$ at $\tilde x$, and since $D v(\tilde x) \neq 0$, we arrive at
$$
0 > |D v(\tilde x)|^\gamma \Delta^s v(\tilde x) \geq 0,
$$
from which we arrive at a contradiction.

For $N = 1$, the same proof follows the same lines above in the case $2s < 1$. For $2s > 1$, we proceed as before by considering the function
$$
v(x) = |x|^{2s - 1} - r^{2s - 1}, \quad x \in \R.
$$

Finally, for $N = 1$ and $s=1/2$, we consider 
$$
v(x) = \max \{ \log (|x|), \log (\epsilon) \} - \log(r), \quad x \in \R,
$$
see section 3 in~\cite{FQ}. The proof is now complete.
\end{proof}

%%%%%%%%%%
\section{Local $C^{1,\alpha}$-regularity estimates}\label{secreg}

In this section, we derive sharp regularity estimates for solutions of \eqref{eq1}. For $s\in (\frac{1}{2},1)$, and $\mathcal L_0$ as in \eqref{elliptic}, we restrict our analysis to the subclass of kernels $\mathcal L_\star \subset \mathcal L_0$, given by kernels $K$ with the form
\begin{equation}\label{LATAM}
K(z) = \frac{a(\hat z)}{|z|^{N + 2s}}, \quad \mbox{for} \quad \hat z = \frac{z}{|z|},
\end{equation}
where $z \in  \mathbb{R}^N\setminus \{0\}$ and $a \in L^\infty(S^{N - 1})$ is a symmetric function satisfying the ellipticity condition $\lambda \leq a \leq \Lambda$. We easily observe that family $\mathcal{L}_\star$ is related to operators $L_K$ which are linear, $2s$-homogeneous operators, and they are associated with stable stochastic L\'evy processes, see~\cite{ROS}.

For a given $q \in \R^N$, we say that a exterior data $g :\mathbb{R}^N\setminus 
\Omega \to \R$ belongs to $\mathcal P(q,\Omega)$ if $g+q\cdot x \in \mathcal{P}(\Omega)$, in sense of  \eqref{RIO}. More precisely, $g \in \mathcal{P}(q,\Omega)$ is equivalent to assuming that problem
\begin{equation}\label{eqPqsg}
\left\{ 
\begin{array}{rll} 
|Du + q|^\gamma I u & = 0 & \mbox{in } \; \Omega \\ 
u & = g & \mbox{in } \; \mathbb{R}^N\setminus \Omega,  
\end{array} 
\right.
\end{equation}
has a unique solution. Consequently, $u$ is $I$-harmonic, i.e. $Iu=0$ in $\Omega$, with $u=g$ in $\mathbb{R}^N\setminus \Omega$. From this, we denote by $\alpha_K$ the gradient H\"older regularity exponent for $I$-harmonic functions,  see \cite{CSARMA,K}.

\medskip

The main result of this section is the following:

\begin{theorem}\label{teoreg}
Let $f \in L^\infty(\Omega)$. For $g \in C(\mathbb{R}^N \setminus \Omega)$, $\sigma \in (0,2s)$ and $M > 0$, assume that $|g(x)| \leq M(1 + |x|)^\sigma$ for $x \in \mathbb{R}^N \setminus \Omega$, and $g \in \mathcal{P}(\Omega)$. Then, viscosity solutions for problem \eqref{eq1} are locally $C^{1,\alpha}(\Omega)$ for
\begin{equation}\label{alphaK}
\alpha \in \left(0,\alpha_K \right) \cap \left( 0, \frac{2s-1}{1+\gamma} \right]. 
\end{equation}
\end{theorem}

In order to prove the result above, we start with the following lemma, which states invariance of uniqueness property under dilatation and addition of linear functions.

\begin{lemma}[Invariance under affine functions]\label{lemaSC}
Assume $g \in C(\mathbb{R}^N \setminus \Omega) \cap L^1_{2s}(\mathbb{R}^N \setminus \Omega)$ and, for some $q\in \mathbb{R}^N$, $g \in \mathcal P(q,\Omega)$. Then, for each parameters $\rho, \lambda > 0$, $a \in \R$ and $b \in \R^N$, we have that 
$$
\rho(g(\lambda x) - b (\lambda x) - a) \in \mathcal P(\rho \lambda(b + q),\lambda^{-1} \Omega).
$$
\end{lemma}

\begin{proof}
The result is a consequence of the following slightly more general result: let $\rho, \lambda > 0$ and $\ell(x) = bx + a$ be a linear function. Then, the function $u$ is a viscosity solution to~\eqref{eqPqsg} if, and only if function
$$
\tilde u(x) = \rho (u(\lambda x) - \ell(\lambda x)),
$$
solves
\begin{equation}\nonumber
\left\{ 
\begin{array}{rlll} 
|Du(x) + \rho \lambda(b + q)|^\gamma I u(x) & = & 0 \quad & \mbox{for } \; x \in \lambda^{-1} \Omega, \\[0.2cm] 
u(x) & = & \rho (g(\lambda x) - \ell(\lambda x)) \quad & \mbox{for } \; x \in \mathbb{R}^N \setminus \lambda^{-1} \Omega .
\end{array} 
\right.
\end{equation}
The above result is a consequence of the following facts: for $s>1/2$, linear functions are $I$-harmonic; the kernels are symmetric; the operator $I$ is positively $2s$-homogeneous. These properties allow us to justify the above computations in the viscosity sense, in which the uniqueness property lies in applying that $g \in \mathcal P(q,\Omega)$. 
\end{proof}

We continue with the following consequence of Proposition~\ref{propestmax}.

\begin{lemma}[Stability property]\label{lemaestP}
Let $M > 0$ and $\sigma \in (0,2s)$. Consider sequences $(q_k)_k \subset \R^N$ and $(g_k)_k \subset C(\mathbb{R}^N \setminus \Omega)$, satisfying $|g_k(x)| \leq M(1 + |x|)^\sigma$, such that $q_k \to q \in \R^N$, and $g_k \to g$ locally uniformly in $\mathbb{R}^N \setminus \Omega$. Then, 
$$
g_k \in \mathcal P(q_k,\Omega) \quad \mbox{implies that}  \quad g \in \mathcal P(q,\Omega).
$$ 
\end{lemma}

\begin{proof}
From Lemma \ref{lemaSC}, for each integer $k>0$, we can assume $q_k = 0$, since $u$ solves~\eqref{eqPqsg} if, and only if $u(x) - q_k x \in \mathcal{P}(g - q_k x, \Omega)$. Hence, even replacing $g_k$ by $g_k - q_k x$, the stability assumptions of the lemma still hold. By Proposition~\ref{propestmax}, we have stability for extremal solutions (see Theorem \ref{teomaxmingeral}), and so, since $g_k \in \mathcal P(0,\Omega)$, we conclude that the extremal solutions associated with the exterior data $g$ coincides with the unique $I$-harmonic function.
\end{proof}

\begin{lemma}[Extension property]\label{lemma extension}
Assume $g \in \mathcal P(q,\Omega)$, and let $u$ be the $I$-harmonic function in $\Omega$ with exterior data $g$ in $\mathbb{R}^N\setminus \Omega$. Then, for each $\Omega' \subset \subset \Omega$, the function $v = u |_{\mathbb{R}^N \setminus \Omega'}$ satisfies $v \in \mathcal P(q,\Omega')$.
\end{lemma}

\begin{proof}
Assume by contradiction, that problem
\begin{equation}\nonumber
\left \{ 
\begin{array}{rlll}
|D\omega + q|^\gamma I \omega & = & 0 \quad & \mbox{in } \; \Omega' \\
\omega & = & v \quad & \mbox{in } \; \mathbb{R}^N\setminus \Omega',  
\end{array}
\right.
\end{equation}
has no unique solution. Hence, from Theorem \ref{teomaxmingeral}, we can assume that the associated maximal solution (and/or the minimal solution), denoted by $\bar u'$ ($\underline{u}'$), differs from the unique $I$-harmonic function in $\Omega'$ with exterior data $v$. Precisely, we have $\bar u' > u$ in $\Omega'$, where $u$ is the $I$-harmonic function with $u=v$ in $\mathbb{R}^N\setminus \Omega'$.

We now claim that function $\bar u'$ is a viscosity subsolution for problem~\eqref{eqPqsg}. In fact, we easily observe that one holds for $x \in \Omega'$. On the other hand, for $x \in \Omega \setminus \Omega'$, we have $\bar u'(x) = u(x)$, and for each $K \in \mathcal K$,
\begin{equation}\nonumber
\begin{array}{rlll}
L_K \bar u'(x) & = & \displaystyle\int_{\Omega'} [\bar u'(y) - \bar u'(x)] K(x - y)dy + \int_{\mathbb{R}^N \setminus \Omega'} [u(y) - \bar u'(x)] K(x - y)dy \\[0.4cm]
& \geq & \displaystyle \int_{\Omega'} [u(y) - u(x)] K(x - y)dy\; + \int_{\mathbb{R}^N \setminus \Omega'} [u(y) - u(x)] K(x - y)dy \\[0.6cm]
& = & L_K u(x).
\end{array}
\end{equation}
From this, we conclude that $\bar u'$ is a subsolution for the problem solved by $u$. Therefore, since $\bar u' > u$ in $\Omega'$, we have the maximal solution for~\eqref{eqPqsg} is different from the $I$-harmonic one, which is a contradiction.
\end{proof}

The last preliminary result is the following approximation lemma.

\begin{lemma}[Flatness improvement]\label{approximation lemma}
Assume $\overline{\alpha} \in (0, 2s - 1)$, $M>0$, and $\omega: [0,+\infty) \to \R_+$ a modulus of continuity. Given $\epsilon \in (0,1)$, there exist universal parameters $R$ large and $\delta$  small such that, for each $q \in \R^N$ and $u$ solving 
$$
-\delta \leq |Du + q|^\gamma I u \leq \delta \quad \mbox{in} \ B_1,
$$
with exterior condition $u = g$ in $\mathbb{R}^N \setminus B_1$, such that
\begin{itemize}	
\item[$(i)$] $g \in \mathcal P(q,B_1)$;

\medskip

\item[$(ii)$] $|g(x)-g(y)|\leq \omega(|x-y|) \quad \mbox{for} \ x,y \in B_R \setminus B_1$;

\medskip

\item[$(iii)$] $|u(x)|\leq M(1 + |x|^{1+\overline{\alpha}}) \quad \mbox{for } x \in \mathbb{R}^N \setminus B_1$,
\end{itemize}
there exists a $I$-harmonic function $h$ in $B_1$ such that
$$
\sup_{B_{1}}|u-h|\leq \epsilon.
$$
\end{lemma}

\begin{proof}
By contradiction, we assume there exists $M, \epsilon > 0$, $\overline{\alpha} \in (0,2s - 1)$, a modulus of continuity $\omega$, sequences $R_k \to +\infty$, $\delta_k \to 0$, $q_k \in \R^N$,  sequence of functions $g_k$, satisfying the assumptions $(i)$-$(iii)$ above, and sequence $u_k$ solving
$$
-\delta_k \leq |Du + q_k|^\gamma I u \leq \delta_k \quad \mbox{in} \ B_1,
$$
with exterior data $g_k$. However, one holds
\begin{equation}\label{contra}
\sup_{B_{1}}|u_k-h| > \epsilon,
\end{equation}
for every $h$ solving $Ih = 0$ in $B_1$. 

In case $q_k \to +\infty$, from Lemma~\ref{lemaSC}, we can normalize the equations and the exterior data without modifying the assumptions above. Thus, we can only assume that $(q_k)_k$ is bounded, and up to a subsequence $q_k \to q_\infty \in \mathbb{R}^N$.

Then, by stability of viscosity solutions, we have that sequence $u_k$ converges to a solution $u_\infty$ to the problem \eqref{eqPqsg} with $g_\infty$ the  local uniform limit of $g_k$. Therefore, from Lemma~\ref{lemaestP}, we conclude that $g_\infty \in \mathcal{P}(q_\infty,B_1)$, and so $u_\infty$ is $I$-harmonic, which contradicts \eqref{contra}. 
\end{proof}

From the results above and the iteration method to be argued below, we prove a discrete version of Theorem~\ref{teoreg}. We highlight the use of uniqueness invariance for rescaled functions in dyadic balls.

\begin{proposition}[Discrete oscillation estimates] \label{iterative teo}
Under assumptions of Lemma \ref{approximation lemma}, there exist universal $\rho \in (0,1)$ and $C > 0$ such that, for each continuous viscosity solution $u$ for problem \eqref{eq1}, with $\|u\|_{L^\infty(\mathbb{R}^N)} \leq 1$ and exterior condition $u = g$ in $\mathbb{R}^N \setminus B_1$, satisfying $(i)$ in the later Lemma, there holds
\begin{equation}\label{aracaju2}
\sup_{B_{\rho^k}}|u-l_k|\leq \rho^{k(1+\alpha)},
\end{equation}
for $l_k=a_k+p_kx$ and $\alpha$ as in \eqref{alphaK}, where
\begin{equation}\label{estimatescoeficientpoly}
\left\{
\begin{array}{rcl}
|a_{k+1}-a_k| & \leq & C\rho^{(1+\alpha)k},\\[0.2cm] 
|b_{k+1}-b_k| & \leq & C\rho^{\alpha k}. 
\end{array}
\right. 
\end{equation}	
\end{proposition}

\begin{proof}
For the sake of convenience, we assume that $u$ solves \eqref{eq1} in the viscosity sense, for $\|f\|_{L^\infty} \leq \delta$ with $\Omega=B_{2R}$, where $\delta$ and $R$ shall be chosen as in the Lemma \ref{approximation lemma}. Also, we assume that $u(0)=0$.

We argue inductively. Denoting $l_0 = 0$ and $w_0=u$, for integers $k>0$, we shall construct a sequence of linear functions $l_k x = a_k + p_k x$, for $a_k \in \R, \ p_k \in \R^N$, and a sequence of functions $w_k$ defined as follows:
\begin{equation}\label{wk+1}
w_{k}(x) = \frac{u(\rho^{k}x) - l_k(\rho^{k}x)}{\rho^{k(1 + \alpha)}},
\end{equation}
for $\rho \in (0,1)$ to be chosen universally, and $\alpha$ as in \eqref{alphaK}. We easily see that $w_{k}$ solves
\begin{equation*}
|D w_k + \rho^{-\alpha k}p_k|^{\gamma}I(w_k) =\rho^{(-\alpha \gamma +2s-1-\alpha)k}f=:f_k \quad \mbox{in } \; B_{\rho^{-k}},
\end{equation*}
and
$$
g_k:=w_k=\rho^{-k(1+\alpha)}(g(\rho^kx)- l_k(\rho^{k}x)) \; \mbox{ in } \; \mathbb{R}^N \setminus B_{\rho^{-k}}.
$$
Hence, from Lemma \ref{lemaSC}, $g_k \in \mathcal P_{\rho^{-k\alpha}p_k}(B_{\rho^{-k}})$, and so, from Lemma \ref{lemma extension}, we obtain  
\begin{equation}\label{cond1}
g_k \in \mathcal P_{\rho^{-k\alpha}p_k}(B_1),
\end{equation}
From this, condition $(i)$ in Lemma \ref{approximation lemma} is satisfied.
In addition, since $\alpha\leq \frac{2s-1}{1+\gamma}$,
\begin{equation}\label{cond2}
\|f_k\|_{L^\infty}\leq \|f\|_{L^\infty} \leq \delta.
\end{equation}

Now, corresponding to condition $(iii)$ in Lemma \ref{approximation lemma}, assume that sequence $\{w_k\}$ satisfies
\begin{equation}\label{aracaju}
|w_k(x)| \leq 1+|x|^{\overline{\alpha}}, \quad \mbox{for } \;  x \in \mathbb{R}^N,
\end{equation}
and integers $k>0$, where 
\begin{equation}\nonumber
\left \{ 
\begin{array}{cl}  \overline{\alpha}=\alpha_K & \mbox{if}\quad \alpha_K\leq\dfrac{2s-1}{1+\gamma}, \\[0.3cm] 
\dfrac{2s-1}{1+\gamma}<\overline{\alpha}<\min\{\alpha_K,2s-1\} & \mbox{if}\quad \dfrac{2s-1}{1+\gamma}< \alpha_K. 
\end{array} 
\right .
\end{equation}

In case \eqref{aracaju} holds for a certain $k>0$, from Theorem \ref{Local Holder estimates}, we have that $w_k$ is uniformly Holder continuous in $B_R\setminus B_1$. Hence, condition $(ii)$ in Lemma \ref{approximation lemma} is satisfied for $g=w_k$. On the other hand, for each $k>0$, we can find function $h_k \in C(B_1)$, satisfying $\| h_k \|_{L^\infty(B_1)} \leq 4$ and 
\begin{equation}\label{harmonica}
\left \{ 
\begin{array}{rll}  I h_k & = 0 \quad & \mbox{in} \ B_1, \\[0.2cm] 
h_k & = w_k\quad & \mbox{in} \ \mathbb{R}^N\setminus B_1. 
\end{array} \right .
\end{equation}
Moreover, for some universal constant $\bar A > 0$, we obtain  
\begin{equation*}
\max\{|h_k(0)|,|Dh_k(0)|\} \leq \bar A,
\end{equation*}
and
\begin{equation*}
|h_k(x) - h_k(0) - Dh_k(0)\cdot x| \leq \bar A |x|^{1 + \overline{\alpha}} \quad \mbox{for each } \; x\in B_1,
\end{equation*}
see~\cite{CSARMA,K}. Since $\alpha<\alpha_K$, we select $\rho>0$ small enough for obtaining
\begin{equation}\label{rho}
\rho^{\overline{\alpha} - \alpha}(4 +  3\bar A) \leq \frac{1}{100}.
\end{equation}
In the sequel, we choose $R$ large and $\delta$ small, such that for each $w_k$ satisfying (\ref{aracaju}), one holds 
$$
\| w_k- h_k \|_{L^\infty(B_1)} \leq \rho^{1 + \alpha}/2.
$$

Next, we choose $a_k$ and $p_k$ inductively. Once $w_k$ is obtained and satisfies~\eqref{aracaju}, for $h_k$ as above, we select $\tilde a_k = h_k(0)$ and $\tilde p_k = D h_k(0)$, consider $\tilde l_k(x) = \tilde a_k + \tilde p_k x$ and define
\begin{equation}\label{lk}
l_{k + 1}(x) := l_k(x) + \rho^{k(1 + \alpha)} \tilde l_k(\rho^{-k}x), \quad x \in \R^N.
\end{equation}
Equivalently, we can denote
$$
a_{k + 1} := a_k + \rho^{1 + \alpha} \tilde a_k \quad \mbox{and} \quad p_{k + 1} := p_k + \rho^{\alpha} \tilde p_k.
$$
Starting with choices $w_0 = u$ and $l_0 = 0$, in view of definition~\eqref{wk+1}, the sequences $\{ w_k \}, \{ l_k \}$ are well-defined. Since \eqref{aracaju} is obtained, then condition \eqref{aracaju2} immediately holds. Also, assuming $C = \bar A$ in \eqref{lk}, we derive \eqref{estimatescoeficientpoly}.

\medskip

To conclude the proof, we show the validity of estimate \eqref{aracaju}. In fact, assume that the latter holds for some $k$. Define
$$
w_{k + 1}(x) := \frac{w_k(\rho x) - \tilde l_k (\rho x)}{\rho^{1 + \alpha}}.
$$ 
Consider $|x| \rho \geq 1$. Since $\alpha<\overline{\alpha}$, we conclude
\begin{equation}\nonumber
\begin{array}{rcl}
|w_{k + 1}(x)| & \leq & \rho^{-(1 + \alpha)} (|w_k(\rho x)| + |\tilde l_k(\rho x)|) \\[0.2cm]
& \leq & \rho^{-(1 + \alpha)} (1 + \rho^{1 + \overline{\alpha}}|x|^{1 + \alpha_K}) + \rho^{-(1 + \alpha)}\bar A (1 + \rho |x|) \\[0.2cm]
& \leq &  \rho^{ \overline{\alpha}-\alpha}(4 + 3\bar A) |x|^{1 +  \overline{\alpha}}.
\end{array}
\end{equation}
Hence, choosing $\rho$ as in~\eqref{rho}, we conclude $|w_{k + 1}(x)| \leq 1+ |x|^{1 + \overline{\alpha}}$. Assuming that $|x| \rho < 1$, we see that
\begin{equation}\nonumber
\begin{array}{rcl}
|w_{k + 1}(x)| & \leq & \rho^{-(1 + \alpha)} (|w_k(\rho x) - h_k(\rho x)| + |h_k(\rho x) - \tilde l_k(\rho x)|) \\[0.2cm]
& \leq & \rho^{-(1 + \alpha)}(\rho^{1 + \alpha}/2 + \bar A |\rho x|^{1 +  \overline{\alpha}}) \\[0.2cm]
& \leq & 1/2 + \bar A \rho^{\overline{\alpha} - \alpha}|x|^{1 + \overline{\alpha}},
\end{array}
\end{equation}
and again by~\eqref{rho}, we conclude the result.
\end{proof}

%%%%%%%%%%
\appendix
%%%%%%%%%%
\section{Vanishing viscosity approximation}\label{existenceapx}

Here, we prove existence and uniqueness for~\eqref{eqeps}. The arguments follow along the lines of proof of Proposition~\ref{comparison}, and for this reason, we will be sketchy in this part.

We start with comparison principle. Let $u \in USC(\R^N) \cap L^1_{2s}(\R^N)$ viscosity subsolution to~\eqref{eqeps}, and $v \in LSC(\R^N) \cap L_{2s}^1(\R^N)$ viscosity supersolution to~\eqref{eqeps}, both locally bounded in $\R^N$ and such that $u \leq v$ in $\mathbb{R}^N\setminus \Omega$. 

Assume 
$$
M:= \sup_\Omega \{ u - v \} = \max_\Omega \{ u - v \} = (u - v)(x_0) > 0.
$$

The existence of $x_0 \in \Omega$ attaining $M$ is a consequence of the upper semicontinuity of $u-v$, the boundedness of $\Omega$ and the fact that $u \leq v$ in $\mathbb{R}^N\setminus \Omega$. Moreover, we have the existence of $d_0 > 0$ such that, for each $x_0$ attaining $M$, we have $d(x_0) > d_0$.

For $\alpha > 0$, consider 
$$
\Phi(x,y) = u(x) - v(y) - \frac{1}{2\alpha^2}|x - y|^2 \quad x, y \in \Omega.
$$

By standard arguments in the viscosity theory, we have the existence of $(\bar x,\bar y) \in \Omega \times \Omega$, maximum point for $\Phi$, such that
$$
\alpha^{-2} |\bar x - \bar y|^2 \to 0, \ u(\bar x) - v(\bar y) \to M, 
$$
and $\bar x, \bar y \to \hat x \in \Omega$ as $\alpha \to 0$, with $\hat x$ attaining $M$.

Then, we can use the viscosity inequalities for $u$ at $\bar x$ and for $v$ at $\bar y$. Using the same analysis as in the proof of Proposition~\ref{comparison}, more specifically the inequalities~\eqref{testing}, we have that for each $a > 0$ and $\delta \in (0, d_0/2)$, there exists $i \in \I, j \in \J$ such that
\begin{equation*}
\begin{split}
& (\epsilon + |\bar p|^\gamma) (\alpha^{-2} o_\delta(1) + L_{ij}[B_\delta^c] u(\bar x)) \geq f(\bar x) - a/2, \\
& (\epsilon + |\bar p|^\gamma) (\alpha^{-2} o_\delta(1) + L_{ij}[B_\delta^c] v(\bar y))  \leq f(\bar y) + a/2,
\end{split}
\end{equation*}
where $\bar p = \alpha^{-2}(\bar x - \bar y)$, and $o_\delta(1) \to 0$ as $\delta \to 0$ uniformly in the rest of the parameters.

Subtracting both inequalities we arrive at
\begin{equation*}
f(\bar x) - f(\bar y) - a \leq (\epsilon + |\bar p|^\gamma) (\alpha^{-2}o_\delta(1) + L_{ij}[B_\delta^c] u(\bar x) - L_{ij}[B_\delta^c] v(\bar y))
\end{equation*}

Then, by similar arguments as in Proposition~\ref{comparison} (see the definitions of $I_1,...,I_4$ there), we arrive at
\begin{equation*}
f(\bar x) - f(\bar y) - a \leq (\epsilon + |\bar p|^\gamma) (\alpha^{-2}o_\delta(1) + o_\alpha(1) - \lambda M \int_{\mathbb{R}^N\setminus \Omega - \hat x} |z|^{-(N + 2s)}dz),
\end{equation*}
where $o_\alpha(1) \to 0$ uniformly in the rest of the parameters (here we use that $d(\hat x) \geq d_0$).

Thus, letting $\delta \to 0$ and taking $\alpha$ small enough in terms of $M, \lambda$ and $d_0$, we arrive at
\begin{align*}
f(\bar x) - f(\bar y) - a & \leq -\lambda M(\epsilon + |\bar p|^\gamma) \int_{\mathbb{R}^N\setminus \Omega - \hat x} |z|^{-(N + 2s)}dz \\
& \leq -\lambda M \epsilon\int_{\mathbb{R}^N\setminus \Omega - \hat x} |z|^{-(N + 2s)}dz,
\end{align*}
and the result follows by taking $\alpha \to 0$, the continuity of $f$ and the fact that $a > 0$ is arbitrary.

The barriers provided in the proof of Proposition~\ref{comparison} (see~\eqref{barrierV}), together with comparison principle and Perron's method lead us to the existence and uniqueness for~\eqref{eqeps}. \qed

\bigskip

{\small \noindent{\bf Acknowledgments.} D.J.A. is partially supported by Conselho Nacional de Desenvolvimento Científico e Tecnológico grant 311138/2019-5 and Paraíba State Research Foundation (FAPESQ) grant 2019/0014. D.J.A. thanks the Abdus Salam International Centre for Theoretical Physics (ICTP) and UFS-Brazil for their hospitality during his research visits. D.P.  was partially supported by Capes-Fapitec, CNPq grant 305680/2022-6 and  Fondecyt Grant No. 1201897.  D.P. thanks the Department of Mathematics of UFPB-Brazil and the Department of Mathematics of USACH-Chile for providing an excellent working environment during the development of this work. E. T. is partially supported by Fondecyt Grant No. 1201897. E.T. thanks Department of Mathematics of UFPB-Brazil, and Department of Mathematics of UFS-Brazil for their hospitality during the research visits to these institutions.}

\medskip

\bibliographystyle{amsplain, amsalpha}

%%%%%

\end{document}